\documentclass[12pt,oneside]{amsart}

\usepackage{amssymb, amsthm, hyperref}
\usepackage{enumerate, amsfonts, amsxtra,amsmath,latexsym,epsfig, color}
\usepackage{mathrsfs, MnSymbol}
\usepackage{geometry}                		
\usepackage{graphicx}
\usepackage{hyperref}				
\usepackage{tikz}								
\usepackage{autobreak}
\usepackage{color}
\usepackage{amsmath,amscd}	
\usepackage{mathtools}
\usepackage{stmaryrd}
\usepackage{animate}
\usepackage{tikz-cd}
\usepackage{soul}

\usepackage{hyperref}

\usepackage{xypic}

\input xy
\xyoption{all}

\newtheorem {theorem}{Theorem}[section]
\newtheorem {lemma} [theorem] {Lemma}
\newtheorem {proposition} [theorem] {Proposition}
\newtheorem {corollary} [theorem] {Corollary}

\newtheorem {notation} [theorem] {Notation}

\theoremstyle{definition}
\newtheorem{definition}[theorem]{Definition}
\newtheorem{remark}[theorem]{Remark}
\newtheorem{example}[theorem]{Example}

\begin{document}
\title[]{Conditions on the monodromy for a surface group extension to be CAT(0)}
\author{Kejia Zhu}
\address{Department of Mathematics, Statistics and Computer Science, University of Illinois at Chicago}
\email{kzhumath@gmail.com}

\date{}
\maketitle

\begin{abstract}
In order to determine when surface-by-surface bundles are non-positively curved, Llosa Isenrich and Py in \cite{Py} give a necessary condition: given a surface-by-surface group $G$ with infinite monodromy, if $G$ is CAT(0) then the monodromy representation is injective. We extend this to a more general result: Let $G$ be a group with a normal surface subgroup $R$. Assume $G/R$ satisfies the property that for every infinite normal subgroup $\Lambda$ of $G/R$, there is an infinite finitely generated subgroup $\Lambda_0<\Lambda$ so that the centralizer $C_{G/R}(\Lambda_0)$ is finite. We then prove that if $G$ is CAT(0) with infinite monodromy, then the monodromy representation has a finite kernel. This applies in particular if $G/R$ is acylindrically hyperbolic. 



\end{abstract}

\section{Introduction:}
Surface bundles are a classical topic, as is the question of understanding which spaces are non-positively curved. A metric space is called \emph{non-positively curved} if every point has a neighborhood which is CAT(0), in the sense that triangles are no fatter than those in Euclidean space  (see \cite[Definition II.1.2]{Bridson}). This is a generalization of manifolds with non-positive sectional curvature. For a more detailed introduction to CAT(0) spaces, see \cite[Part II]{Bridson}.\\

Now it is natural to consider the non-positively curved condition for the case of surface bundles with fibers of genus greater than $1$, as they are ``extensions" of non-positively curved spaces. Thurston proved that a fibration over $S^1$ with closed higher genus fiber is hyperbolic if and only if the monodromy is pseudo-Anosov (see e.g. \cite[Chapter 6]{otal2001hyperbolization}). Leeb in \cite[Section 3]{Leeb1995} proves many surface bundles over the circle with fibers of genus greater than $1$ are non-positively curved, but Kapovich and Leeb in \cite[Theorem 3.7]{Kapovich1996ActionsOD} show that given a closed hyperbolic hyperbolic surface $S$, if the mapping class defining the monodromy, say $f$, is a collection of Dehn twists (all in the same direction) then the mapping torus of $f$ does not admit a metric of nonpositive curvature. For the case of (hyperbolic) surface bundles over a surface, some of the classical examples are non-positively curved: this is the case of all double étale Kodaira fibrations in the sense of \cite{catanese2009double}  and of the Kodaira-Atiyah examples (see \cite{atiyah2015signature} and \cite{kodaira2015certain}). See \cite[Theorem 3]{Py} for an explanation.  Before Example \ref{ex}, we first introduce Llosa Isenrich and Py's theorem as well as some necessary definitions and facts.

\begin{theorem}\cite[Theorem 3]{Py}\label{2}
	Let $G$ be a surface-by-surface group with infinite monodromy. Fix a normal subgroup $R\unlhd G$ isomorphic to a surface group with $G/R$ isomorphic to a surface group. If the group $G$ is $\mathrm{CAT(0)}$ then the monodromy $G/R\to\mathrm{Out}(R)$ is injective.
\end{theorem}

\begin{remark}
	The fundamental group of an oriented closed surface of genus larger than $1$ is called a \emph{surface group}. A group is called CAT(0) if it admits a properly discontinuous and cocompact action on a CAT(0) space. Let $X$ be a compact non-positively curved space.  It follows from the definition that $\pi_1(X)$ is a CAT(0) group.  However, there are CAT(0) groups that do not arise in this way (for example, some CAT(0) groups have torsion, whereas if $X$ is compact non-positively curved then the Cartan-Hadamard Theorem ( \cite[Theorem II.4.1]{Bridson}) says that $X$ is a $K(\pi,1)$, so $\pi_1(X)$ is torsion-free).
\end{remark}

Given a closed (hyperbolic) surface bundle $F\to E\to B$ of (connected) manifolds, there is a short exact sequence $$1\to\pi_1(F)\to\pi_1(E)\to\pi_1(B)\to 1.$$
Recall the isomorphism classes of the (oriented) $F$-bundles over $B$ are in bijection with the conjugacy classes of homomorphisms $\pi_1(B) \to \text{Mod}(F)$, where $\text{Mod}(F)$ is the mapping class group of $F$ (see \cite[p.155]{farb2011primer}). 
So the group $\pi_1(E)$ is determined by the monodromy representation.


\begin{example}\label{ex}
	It is easy to build a surface bundle over surface $F \to E \to B$, so that the monodromy representation $\psi:\pi_1(B) \to \text{Mod}(F)$ is not injective but has infinite image. Such a $\psi$ is easy to construct. For example, consider the composition $\varphi=\varphi_2\circ\varphi_1$, where $\varphi_1$ is a surjection from $\pi_1(B)$ onto $\mathbb Z$ and $\varphi_2$ is a map from $\mathbb Z\to\text{Mod}(F)$, sending a generator of $\mathbb Z$ to an infinite order element of $\text{Mod}(F)$.  Theorem \ref{2} shows that the $F$-bundle over $B$ corresponding to  $\psi$ is an example of a surface bundle over a surface which is not non-positively curved. 
\end{example}
Motivated by Theorem \ref{2}, in this paper, we consider the following question:  Suppose that $G$ is a group.  When does $G$ have the property that for any surface $F$ of genus at least $2$, and any group extension $E$ defined by a homomorphism $\rho : G \to Mod(F)$ with infinite image, if $E$ is CAT(0) then $\ker(\rho)$ is finite?

\begin{definition}
	A group $Q$ has Property (LIP) if for every infinite normal subgroup $\Lambda \trianglelefteq Q$, there is an infinite finitely generated subgroup $\Lambda_0<\Lambda$ so that the centralizer $C_Q(\Lambda_0)$ is finite.
\end{definition}

The following theorem generalizes Theorem \ref{2}:

\begin{theorem}\label{ge}
	Suppose $R$ is a surface group, that $$1\to R \to G \to\Gamma\to 1$$ is a short exact sequence with infinite monodromy, and that $\Gamma$ has Property (LIP). If $G$ is $\mathrm{CAT}(0)$ then the monodromy representation $\Gamma\to\mathrm{Out}(R)$ has finite kernel.
\end{theorem}

\begin{remark}
	The proof of Theorem \ref{2} by Llosa-Isenrich and Py generalizes easily to prove Theorem \ref{ge} (see Section \ref{s3}). Indeed the proof in \cite{Py} relied solely and implicitly on the fact that surface groups have property (LIP), a fact proved in loc.cit..
\end{remark}

Theorem \ref{ge} raises the question of which groups have Property (LIP).

\begin{example}
	The first examples we give of groups with Property (LIP) are higher rank lattices: Let $G$ be a connected semi-simple Lie group with finite centre with $\text{rk}_{\mathbb R}(G)\ge 2$. Then every irreducible lattice of $G$ satisfies Property (LIP). This follows quickly from the Normal Subgroup Theorem of Margulis \cite{margulis1991discrete}, because every infinite normal subgroup is actually of finite index and hence has finite center, therefore finite centralizer. Some concrete examples are:\\
	\noindent (1) $\text{SL}(n, \mathbb Z)\subset \text{SL}(n, \mathbb R)$ for $n\ge 3$;\\ 
	\noindent (2) $\text{SL}(2, \mathbb Z[\sqrt 2])\subset \text{SL}(2, \mathbb R)\times \text{SL}(2, \mathbb R)$.
\end{example}

Unfortunately, Theorem \ref{ge} applied to higher rank lattices does not give any new information about surface bundles because Farb and Masur in \cite[Theorem 1.1]{farb1998superrigidity} showed that any homomorphism from an irreducible lattice in a semisimple Lie group of higher rank to the mapping class group has finite image.\\

Our main example of groups with Property (LIP) is that of acylindrically hyperbolic groups (see Definition \ref{ac}) as shown by the next result.

\begin{theorem}\label{main}
	Acylindrically hyperbolic groups satisfy Property (LIP).
\end{theorem}

That acylindrically hyperbolic groups satisfy Property (LIP) will not surprise experts, but we give a proof in Section 2.\\

	Every right-angled Artin group is either cyclic, directly decomposable, or acylindrically hyperbolic (see \cite[Appendix-Example (d)]{acylindrically}). It is easy to see cyclic and directly decomposable right-angled Artin groups don't satisfy Property (LIP), so by Theorem \ref{main} the right-angled Artin groups that satisfy Property (LIP) are precisely the acylindrically hyperbolic ones.\\ 

\begin{corollary}\label{1.1} 
Suppose $G$ is a group, with an infinite normal subgroup $R$ which is a closed surface group, and suppose $G/R$ is acylindrically hyperbolic. If $G$ is $\mathrm{CAT}(0)$ then the monodromy $G/R\to\mathrm{Out}(R)$ of the extension is either finite or has finite kernel.
\end{corollary}

Since non-elementary hyperbolic groups are acylindrically hyperbolic, the following example is an application of Corollary \ref{1.1}:\\

\begin{corollary}
	Consider a surface bundle with  infinite monodromy $F\to E \to B$ with $B$ compact, and suppose the universal cover of $B$ is a $\delta$-hyperbolic space. If $\pi_1(E)$ is $\text{CAT}(0)$ and $\pi_1(B)$ is not elementary, then the monodromy representation $\pi_1(B)\to \text{Mod}(F)$ has finite kernel.
\end{corollary}




\noindent\textbf{Acknowledgment:} I would like to thank my advisor, Daniel Groves, for introducing me to the subject and answering my questions. This paper would not have been written without his help. I would like to thank my coadvisor, Anatoly Libgober, for
constant support and warm encouragement. I also wish to thank Claudio Llosa Isenrich and Pierre Py for helpful comments. I am also grateful to the referee(s) for many helpful comments which improved the paper.\\

{\section{Acylindrically hyperbolic groups have Property (LIP) }}\label{1acy}
In this section, we will first introduce the definition of acylindrically hyperbolic groups and some results of Osin from \cite{acylindrically}. Then we will give the proof Theorem \ref{main}.\\

\begin{definition}\label{loxo}
	Given a group $G$ acting on a $\delta$-hyperbolic geodesic space $S$ (\cite{Bridson}-III.H.Definition 1.20), an element $g\in G$ is called \emph{elliptic} if some (equivalently, every) orbit of $g$ is bounded, and \emph{loxodromic} if the map $\mathbb Z\to S$ defined by $n\mapsto  g^n\cdot s$ is a quasi-isometric embedding for some (equivalently, every) $s\in S$. Define $\alpha_g$ to be a geodesic connecting $s$ and $g\cdot s$, it leads to a quasi-geodesic (\cite{Bridson}-I.8.22) $l_g:=\bigcup_{i\in \mathbb Z}g^i\cdot \alpha_g$ called the \emph{axis} of $g$.  Moreover, any loxodromic element $g$ has exactly two fixed points on the Gromov boundary $\partial S$, which are denoted by $g^\pm$.
\end{definition}	

\begin{definition}	
	 Loxodromic elements $g, h\in G$ are called \emph{independent} if the sets $\{g^{\pm}\}$ and $\{h^{\pm}\}$ are disjoint. An action of a group $G$ on a $\delta$-hyperbolic space $S$ is called \emph{elementary} if the limit set of $G$ on $\partial S$ contains
	at most $2$ points.
\end{definition}

\begin{definition}\label{ac}
	An action of a group $G$ on a metric space $S$ is called \emph{acylindrical} if for every $\epsilon> 0$ there exist $R, N > 0$ such that for every two points $x, y$ with $ d(x, y)\ge R$, there are at most $N$ elements
	$g\in G$ satisfying $d(x, g x)\le\epsilon$ and $d(y, g y) \le\epsilon$.
\end{definition}

\begin{theorem}[\cite{acylindrically}-Theorem 1.1]\label{osin1.1}
	Let $G$ be a group acting acylindrically on a $\delta$-hyperbolic space. Then
	$G$ satisfies exactly one of the following three conditions.\\
	(a) $G$ has bounded orbits.\\
	(b) $G$ is virtually cyclic and contains a loxodromic element.\\
	(c) $G$ contains infinitely many (pairwise) independent loxodromic elements.\\
\end{theorem}

	If the action is acylindrical, non-elementarity is equivalent to
	condition (c) from Theorem \ref{osin1.1}. (See \cite[p.852]{acylindrically}).\\

	A group is called \emph{acylindrically hyperbolic} if it admits a non-elementary acylindrical action on a $\delta$-hyperbolic space.  \\

\begin{remark}
	The definition of acylindricity is due to Bowditch \cite{bowditch2008tight}, the earliest version of acylindrically hyperbolicity is given by Sela \cite{sela1997acylindrical} and the definition of acylindrically hyperbolic group is due to \cite{acylindrically}. Some examples of acylindrically hyperbolic groups are:\\
	\noindent (1) Non-elementary hyperbolic groups;\\
	\noindent (2) Non-elementary relatively hyperbolic groups;\\
	\noindent (3) Mapping class groups of almost every surface of finite type;\\ \noindent (4) Non-cyclic directly indecomposable right-angled Artin groups; \\ 
	\noindent (5) The fundamental group of every irreducible, compact
	$3$-manifold with a non-trivial JSJ-decomposition (see \cite{minasyan2015}-Theorem 2.8).
\end{remark}

	

\begin{notation}
	Fix an acylindrically hyperbolic group $G$ and a $\delta$-hyperbolic space $X$ 
	on which $G$ acts acylindrically.

\end{notation}

\begin{lemma}\label{loxod}
	If $\Lambda \unlhd G$ is infinite, then $\Lambda$ contains a loxodromic element, say $g$.
\end{lemma}
\begin{proof}
 By Theorem \ref{osin1.1}, it suffices to show $\Lambda$ contains unbounded orbits. Suppose all the orbits of $\Lambda$ were bounded and take $S$ to be a bounded orbit. Define $\epsilon$ to be the diameter of $S$. By the definition of acylindrical action, there exist $R,N>0$ such that for every $x,y$ with $d(x,y)\ge R$, there are at most $N$ elements $e'\in G$ satisfying 
	$$d(x,e' x)\le\epsilon\quad \textrm{and}\quad d(y,e' y)\le\epsilon.$$
	
	Let $g_1\in G$ be a loxodromic element with $d(g_1\cdot S,S)\ge R$. By definition, there are at most $N$ elements $\{g_i'\}$ so that for every $x\in S$, $y\in g_1\cdot S$,
	$$d(x,g_i' x)\le\epsilon\quad \textrm{and}\quad d(y,g_i' y)\le\epsilon.$$
	
	Since $\Lambda$ is normal, $g_1^{-1} h g_1\in \Lambda$ for each $h\in \Lambda$. Note $h$ and $g_1^{-1} h g_1 $ preserve $S$ by assumption, thus $h$ preserves $S$ and also $g_1\cdot S$ since $h=g_1 g_1^{-1} h g_1 g_1^{-1}$. So for every $x\in S$, $y\in g_1\cdot S$,
	$$d(x,h x)\le\epsilon\quad \textrm{and}\quad d(y,h y)\le\epsilon.$$\\
	
	However, $|\Lambda|=\infty>N$, so we get a contradiction. So $\Lambda$ does not have bounded orbits and hence contains a loxodromic element, say $g$. 
\end{proof}



\begin{lemma}\label{leh0}
Let $g$ a loxodromic element in $\Lambda$, there always exists a loxodromic element, say $a_1\in G$, such that $\{g^{\pm}\}$ and $\{a_1^{\pm}\}$ are disjoint.
\end{lemma}
\begin{proof}
	Suppose there is no such $a_1$. Find a pair of  independent loxodromic elements $k,h$ with $k,h\not=g$. Without loss of  generality, we may assume $h^-=g^-$ and $k^-=g^+$. Consider the loxodromic element $k^h:=hkh^{-1}$, with fixed set $\{h\cdot k^\pm\}$. \\
	
	First we claim that $h^n\cdot k^-\not\in \{g^\pm\}$ holds for any $n\in \mathbb Z\setminus \{0\}$: Suppose $h^n\cdot k^-=g^-$, then by $g^-=h^-$, it follows that $k^-=h^{-n}\cdot h^-=h^-$, contradiction. Suppose $h^n\cdot k^-=g^+$, then by $k^-=g^+$, $h^n\cdot k^-=k^-$, thus $k^{-1}\in \{h^\pm\}$, contradiction. So the claim is proved. In particular, $h\cdot k^-\not\in \{g^\pm\}$.\\

	Now we claim that  for any $n\in\mathbb Z$, $h^n\cdot k^+\not=g^-$.  Suppose $h^n\cdot k^+=g^-$, then since $h^-=g^-$, $h^n\cdot k^+=h^-$, thus $k^+=h^-$, contradiction. \\
	
   If $h\cdot k^+=g^+$, then $h\cdot k^+=k^-$. We claim $h^2\cdot k^+\not=g^+$. Otherwise $h\cdot g^+=h\cdot h\cdot k^+=g^+$ and $g^+\in \{h^\pm\}$, contradiction. \\

	So either $g$ and $k^h$ are independent or $g$ and $k^{h^2}$ are independent and the proof is complete. 
\end{proof}

\begin{lemma}\label{leh1}
	Let $g$ a loxodromic element in $\Lambda$, there exists a loxodromic element $a\in G$ such that $g$, $a g a^{-1}$ are independent. 
\end{lemma}
\begin{proof}
	By the Lemma \ref{leh0}, we know there always exists a loxodromic element $a_1$ such that $\{g^{\pm}\}$ and $\{a_1^{\pm}\}$ are disjoint. \\

	We claim that either $\{a_1\cdot g^\pm\},\{g^\pm\}$ are disjoint or $\{a_1^2\cdot g^\pm\},\{g^\pm\}$ are disjoint. If $a_1\cdot g^+=g^-$, then ${a_1}^2\cdot g^+\not=g^-$ since otherwise it implies $a_1\cdot g^-=g^-$, which violates the assumption that $a_1$ has only two fixed points on the Gromov boundary. Observe ${a_1}^2$ is a loxodromic element with fixed set $\{a_1^{\pm}\}$ on the Gromov boundary. It follows ${a_1}^2\cdot g^+\not=g^+$ and ${a_1}^2\cdot g^-\not=g^-$. Also ${a_1}^2\cdot g^-\not=g^+$ since otherwise it implies ${a_1}^3\cdot g^-=a_1\cdot g^+=g^-$ , while ${a_1}^3$ is a loxodromic element with fixed set $\{a_1^{\pm}\}$ on the Gromov boundary. Similarly, if $a_1\cdot g^-=g^+$, then $\{{a_1}^2\cdot g^{\pm}\}$ and $\{g^{\pm}\}$ are disjoint. Now we can conclude that we can always find a loxodromic element $a$ (either $a=a_1$ or $a=a_1^2$) so that $\{a\cdot g^{\pm}\}$ and $\{g^{\pm}\}$ are disjoint.\\ 

	If $g$ is loxodromic then so is its conjugate $a g a^{-1}$. Observe the fixed points of $a g a^{-1}$ are $a\cdot g^{\pm}$,  so we conclude that $g$ and $a g a^{-1}$ are independent.
\end{proof}

\begin{lemma}
	Given $H=\langle g,g^a\rangle$, the centralizer $C_G(H)$ is a finite group.
\end{lemma}
\begin{proof}
	By \cite[Theorem 6.9]{acylindrically}, since $g$ is loxodromic, $C_G(\langle g\rangle)$ is virtually cyclic, thus the infinite cyclic subgroup $\langle g\rangle$ is of finite index in $C_G(\langle g\rangle)$. For the same reason, $\langle g^a \rangle$ is of finite index in $C_G(\langle g^a\rangle)$. Since $H=\langle g,g^a\rangle$, $$C_G(H)\subset C_G(\langle g\rangle)\cap C_G(\langle g^a\rangle).$$ Also note that the intersection of two virtually cyclic groups is either finite, or else it is virtually cyclic and it contains a common power of $g$ and $g^a$. However, $g$ and $g^a$ are two loxodromic elements with different fixed point sets, so $g$ and $g^a$ do not have common powers.
	Thus $C_G(\langle g\rangle)\cap C_G(\langle g^a\rangle)$ is finite, hence $C_G(H)$ is finite.
\end{proof}

	

\begin{corollary}\label{5}
	Let $Q$ be an acylindrically hyperbolic group and $\Lambda$ be an infinite normal subgroup of $Q$. Then there exists an infinite finitely generated subgroup $\Lambda_0<\Lambda$, so that $C_{Q}(\Lambda_0)$ is finite.\\
	
\end{corollary}

The previous corollary finishes the proof of Theorem \ref{main}.\\

\section{Proof of Theorem \ref{ge}}\label{s3}
Now we start the proof of  Theorem \ref{ge}. It follows from the proof of \cite[Theorem 3]{Py} with only minor changes. The proofs from \cite{Py} of Lemma \ref{30}, Proposition \ref{31} and Proposition \ref{33} below work in our setting exactly as written. We omit some details which are included in \cite{Py}, the reader may find them in Section 3 of \cite{Py}. 

Suppose we are given an exact sequence
$$1 \to R \to G \xrightarrow{p} Q \to 1,$$ where $R$ is a surface group, $G$ is CAT(0), $Q$ has Property (LIP) and the natural monodromy morphism $\phi: Q \to \text{Out}(R)$ is infinite. We consider the centralizer of $R$ in $G$, denoted by $\Lambda$. The proof is by contradiction, so assume $\phi$ has infinite kernel.\\


\begin{lemma}\cite[Lemma 30]{Py}\label{30}
	Let $G$ be a group. Assume that $G$ has a normal subgroup $R$ which is a surface group and let $p: G\to G/R$ be the quotient morphism. Then the centralizer $\Lambda$ of $R$ in $G$ is normal in $G$. The restriction of $p$ to $\Lambda$ is an isomorphism onto the
	kernel of the monodromy morphism $G/R\to \mathrm{Out}(R)$.
\end{lemma}

\begin{remark}
	Recall $R$ is a surface group, which has trivial center. So the subgroup of $G$ generated by $R$ and $\Lambda$ is isomorphic to $R \times\Lambda$.
\end{remark}


We fix a properly discontinuous and cocompact action $G\curvearrowright(E, d)$, where $(E, d)$ is a proper CAT($0$) space.  Note in our case, the $G$ here might not act faithfully on $E$. Observe the kernel of the action of $G$ on $E$ is a finite normal subgroup of $G$, say $K$.  In Lemma \ref{FiniteQuotient}, we show $Q/p(K)$ still has Property (LIP). So we may freely pass $G$, $Q$ to the quotients $G/K$, $Q/p(K)$ respectively. We abuse the notations by still denoting the quotients by $G$, $Q$ and $\Lambda$ respectively. Now we split into two cases:\\

\noindent\textbf{Case 1:} Suppose the group $R\times\Lambda$ does not fix any point in the visual boundary of $E$.\\

As \cite{Py} pointed out, by \cite{monod}, in this case there exists a closed $R \times\Lambda$-invariant convex subset $M \subset E$ which is minimal for these properties and canonical, and which is $G$-invariant. Moreover, the action of $G$ on $M$ is properly discontinuous and cocompact. Also, there exists an isometric splitting
$$M\cong M_1 \times M_2,$$ so that $R$ acts isometrically on $M_1$, $\Lambda$ acts isometrically on $M_2$ and the action of $R\times\Lambda$ on $M$ is the product of these two actions. For details, see p.465 of \cite{Py}.\\



\begin{proposition}\cite[Proposition 31]{Py}\label{31}
	The $G$-action on $M\cong M_1 \times M_2$ is a product action.
\end{proposition}

The previous proposition implies that we can now consider the $G$-action on each $M_i$ separately. It factors through a faithful action of $G/\Lambda$ on $M_1$ (resp. $G/R$ on $M_2$).
\\

\begin{proposition}(cf.\cite{Py}, Proposition 32):\label{32}
	The $G/\Lambda$-action on $M_1$ is properly discontinuous and cocompact. Similarly the $G/R$-action on $M_2$ is properly discontinuous and cocompact.
\end{proposition}
\begin{proof}
	The proof from \cite{Py} that $G/\Lambda$ and $G/R$ act cocompactly and $G/\Lambda$ acts properly discontinuously works in this more general setting as written. Indeed, as the proof from \cite{Py} points out that by \cite[Theorem 5.67]{dructu}, the action of $G/\Lambda$ (resp. $Q\cong G/R$) on $M_1$ (resp. $M_2$) is properly discontinuous if and only if $G/\Lambda$ (resp. $Q$) is discrete in the group $\text{Isom}(M_1)$ (resp. $\text{Isom}(M_2)$).\\

	
	The proof of $G/\Lambda$ is discrete in $\text{Isom}(M_1)$ is the same as \cite{Py}. It remains to show $Q$ is discrete in $\text{Isom}(M_2)$. Since $M_2$ is an invariant subset for the action of $\Lambda$ on $M$, $\Lambda$ acts properly discontinuously on $M_2$. Hence $\Lambda < \text{Isom}(M_2)$ is discrete.  Recall $\Lambda\cong\ker\phi$ is infinite, so by Property (LIP), $\Lambda$ has a finitely generated subgroup $\Lambda_0$ so that $C_Q(\Lambda_0)$ is finite. Then there is a neighbourhood $U$ of the identity in $\text{Isom}(M_2)$ such that every element of $U$ which normalizes $\Lambda$ must centralize $\Lambda_0$.  So $Q\cap U$ is finite and $Q$ is discrete.
\end{proof}

Fix a point $(m_1, m_2)\in M_1\times M_2$. Define $$f:G\to M_1\times M_2,$$ $$g\mapsto (g\cdot m_1, g\cdot m_2).$$
It induces a commutative diagram:\\


\begin{center}
	\begin{tikzcd}
		G\arrow[r,"f"] \arrow[d,"\psi"]
		& M_1\times M_2 \arrow[d,"Id"] \\
		G/\Lambda\times G/R\arrow[r]
		& M_1\times M_2
	\end{tikzcd}
\end{center}


By analyzing the diagram with Proposition \ref{32}, it can be verified that 
$$\psi:G\to G/\Lambda\times G/R \qquad (1)$$
in the diagram is a quasi-isometry. Observe it is indeed injective, thus an injective quasi-isometry between two finitely generated groups, hence its image has finite index. By taking the quotient by the subgroup $R$ on the left and by its image on the right on (1), it gives 
$$\psi':Q\to \phi(Q) \times Q,$$
(recall $\Lambda\cong\ker\phi$ and $G/R\cong Q$). Note the image of $\psi'$ is the graph of $\phi$. If $\text{Im}\psi$ had finite index, the graph of $\phi$ would have finite index in $\phi(Q)\times Q$, but this could only happen if the image of $\phi$ is finite, contradiction.\\

\noindent\textbf{Case 2:} If the group $R\times\Lambda$ has at least one fixed point in the visual boundary of $E$.\\

\begin{proposition} \cite[Proposition 33]{Py}\label{33}
	Let $\Gamma\curvearrowright Z$ be a group acting properly discontinuously and cocompactly on a $\mathrm{CAT}(0)$ space. Let $N<\Gamma$ be a finitely generated subgroup. If $N$ fixes a point in the visual boundary of $Z$, then the centralizer of $N$ in $\Gamma$ is infinite.
\end{proposition}

If $R\times \Lambda$ fixes a point in the visual boundary of $E$, so does the finitely generated group $R\times \Lambda_0$ (recall $R,\Lambda_0$ are both finitely generated). We apply the previous proposition to $N = R\times \Lambda_0$ and obtain that the centralizer $C_G(N)$ of $N$ is infinite. However, $$C_G(N)=C_G(R\times \Lambda_0)=C_G(R)\cap C_G(\Lambda_0)=\Lambda\cap C_G(\Lambda_0)=C_\Lambda(\Lambda_0),$$ which is finite, since by the Property (LIP), $C_\Lambda(\Lambda_0)<C_Q(\Lambda_0)$ is finite. So we deduce a contradiction and this proves Theorem \ref{ge}.\\

The following lemma is applied earlier in this section and the proof is motivated by \cite[Lemma 1.A]{MR1105339}.
\begin{lemma}\label{FiniteQuotient}
	If $G$ is a group with Property (LIP) and $K$ is a finite normal subgroup of $G$, then $G/K$ also has  Property (LIP).
\end{lemma}
\begin{proof}
	Let $H/K<G/K$ be an infinite normal subgroup. 
	By assumption, since $H$ is an infinite normal subgroup of $G$, there is an infinitely finitely generated subgroup $H_0$, so that $C_G(H_0)$ is finite. We claim $C_{G/K}(H_0K/K)$ is finite. Assume $C_{G/K}(H_0K/K)=C/K$, where $C\subset G$. Let $C$ act on $K$ by conjugation, denote the kernel of the action by $C_0$, which is obviously of finite index in $C$. By assumption, for any $h\in H_0, x\in C_0$, $[h,x]:=hxh^{-1}x^{-1}\in K$. In particular, since the elements of $C_0$ commute with the elements of $K$, for any $h\in H_0, x,y\in C_0$, 
	$$[h,x]^{-1}=x^{-1}[h,x]^{-1}x=[h,x]^{-1},$$
	$$[h,xy]=[h,x]x[h,y]x^{-1}=[h,x][h,y].$$
	So given any $h\in H_0$, there is a homomorphism 
	$$\phi_h:C_0\to [h,C_0],$$
	$$x\mapsto [h,x],$$
	where $[h,C_0]:=\{[h,x]|x\in C_0\}$.
	Since $H_0$ is finitely generated, we may assume $H_0$ is generated by $\{h_1,...,h_n\}$. So there is a homomorphism 
	$$\phi: C_0\to \prod_{i=1}^n [h_i,C_0],$$
	$$x\mapsto (\phi_{h_1}(x),...,\phi_{h_n}(x)).$$
	Note $\ker\phi\subset C_G(H_0)$, hence $\ker\phi$ is finite. Also observe, each $[h_i,C_0]\subset K$, thus $[h_i,C_0]$ is finite, which implies $[C:\ker\phi]=[C:C_0][C_0:\ker\phi]<\infty$. Therefore $C$ is finite. In particular, $C_{G/K}(H_0K/K)=C/K$ is finite.
\end{proof}


\bibliography{11} 
\bibliographystyle{alpha} 

\end{document}